\documentclass[12pt]{amsart}

\usepackage{CJK,CJKnumb,amsmath}
\usepackage{amsfonts}
\usepackage{amsthm}
\usepackage{geometry}
\usepackage{mathrsfs}
\usepackage{amsmath}
\usepackage{amssymb}
\usepackage{amsfonts}
\usepackage[percent]{overpic}
\usepackage{bm}
\usepackage[all]{xy}
\usepackage{graphicx}
\usepackage{subfigure}
\usepackage{latexsym}
\usepackage[colorlinks, linkcolor=blue, anchorcolor=bule, citecolor=blue]{hyperref}
\usepackage{epigraph}
\usepackage{hyperref}
\usepackage{fancyhdr}
\usepackage{scalerel}
\usepackage{enumitem} 	
\usepackage{floatrow}
\usepackage{mathtools}
\mathtoolsset{showonlyrefs}

\numberwithin{equation}{section}
\usepackage{color}

\theoremstyle{remark}
\newtheorem{remark}{Remark}[section]

\newtheorem*{ack}{Acknowledgement}
\theoremstyle{plain}% default
\newtheorem{thm}{Theorem}[section]
\newtheorem{thmx}{Theorem}

\newtheorem{satz}{Satz}[section]

\newtheorem{lem}[satz]{Lemma} %[chapter]

\theoremstyle{definition}
\newtheorem{defn}[satz]{Definition} %[chapter]

\theoremstyle{remark}

\numberwithin{equation}{section}

\def\dist{\operatorname{dist}}

\newcommand{\sing}{\operatorname{sing}}

\newcommand{\B}{\operatorname{\mathcal{B}}}

\newcommand{\Log}{\operatorname{Log}}

\newcommand{\C}{\mathbb{C}}				%komplexe Zahlen C
\newcommand{\R}{\mathbb{R}}				%reelle Zahlen R
\newcommand{\N}{\mathbb{N}}				%natürliche Zahlen N
\newcommand{\Z}{\mathbb{Z}}				%ganze Zahlen Z
\newcommand{\D}{\mathbb{D}}				%Einheitskreis
\newcommand{\J}{J}			%Juliamenge von...
\newcommand{\F}{F}			%Fatoumenge von...
\newcommand{\ul}{\underline}					%Abkürzung für \overline

\newcommand{\abs}[1]{\left\lvert#1 \right\rvert}

\newcommand{\gdw}{  \relax  \ifmmode \Longleftrightarrow  \else    $\Longleftrightarrow$  \fi}

\def\Im{\operatorname{Im}}
\def\Re{\operatorname{Re}}

\begin{document}
\title{Nowhere differentiable hairs for entire maps}
\author{Patrick Comd{\"u}hr}

\address{Mathematisches Seminar, Christian-Albrechts-Universit\"at zu Kiel, Ludewig-Meyn-Str. 4, D-24098 Kiel, Germany.}
\email{comduehr@math.uni-kiel.de}


\begin{abstract}
%In 1895 Weierstra{\ss} gave a first example of a continuous function on $\R$, which is nowhere differentiable. Fatou has shown that for $f(z)=z^2+c$, the Julia set $\J(f)$ is a Jordan curve for sufficiently small $\abs{c}\in\R$, but it is nowhere differentiable. 
In 1984 Devaney and Krych showed that for the exponential family $\lambda e^z$, where $0<\lambda <1/e$, the Julia set consists of uncountably many pairwise disjoint simple curves tending to $\infty$, which they called hairs. Viana proved that these hairs are smooth. Bara\'nski as well as Rottenfu{\ss}er, R\"uckert, Rempe and Schleicher gave analogues of the result of Devaney and Krych for more general classes of functions. In contrast to Viana's result we construct in this article an entire function, where the Julia set consists of hairs, which are nowhere differentiable.
\end{abstract}

\maketitle
\section{Introduction and main result} \label{Intro}
% For an entire function $f$ the Julia set $J(f)$ of $f$ is the set of all points in $\C$ where the iterates $f^k$ of $f$ do not form a normal family in the sense of Montel. Given an attracting fixed point $\xi$ of $f$ we denote by $A(\xi):=\{z : \lim_{k\to\infty} f^k(z)=\xi \}$ the \textit{basin of attraction} of $\xi$. From the theory of complex dynamics it is well-known that $J(f)=\partial A(\xi)$, see \cite[Corollary 4.12]{Mi06}. Another set, the so-called escaping set $I(f):=\{z : \lim_{k\to\infty} f^k(z)=\infty \}$, was introduced by Eremenko \cite{Er89} and is of much interest in complex dynamics. For further information on complex dynamics we refer to \cite{Bea91,Ber93,Mi06,St93}.
For the exponential family $f(z)=\lambda e^z$, where $0<\lambda<1/e$, it is not difficult to show that there exists a unique attracting fixed point $\xi \in \R$.
Devaney and Krych \cite{DK84} studied for these functions the basin of attraction of $\xi$, i.e.\  $A(\xi):=\{z : \lim_{k\to\infty} f^k(z)=\xi \}$. %and gave a detailed description of the structure of $J(f)$. 
We say that a subset $H$ of $\C$ is a \textit{hair}, if there exists a homeomorphism $\gamma \colon [0,\infty) \to H$ such that $\lim_{t\to\infty} \gamma(t)=\infty$. We call $\gamma(0)$ the \textit{endpoint} of the hair.

We only state the part of the result due to Devaney and Krych which is relevant for us.
\begin{thmx} \label{ThmA}
	For $0<\lambda <1/e$ the set $\C\setminus A(\xi)$ is an uncountable union of pairwise disjoint hairs.
\end{thmx}

Viana \cite{Vi88} investigated the differentiability of the hairs considered by Devaney and Krych. 

\begin{thmx} \label{ThmD}
	For all $0<\lambda <1/e$ the hairs of $\lambda e^z$ are $C^\infty$-smooth.
\end{thmx}	

The existence of hairs is not restricted to the situation considered by Devaney and Krych. Hairs also appear in the so-called \textit{Eremenko-Lyubich class} $\mathcal{B}$ which has become a major topic of research in the last two decades. It is the class of all transcendental entire functions whose set of critical and asymptotic values, denoted by $\sing(f^{-1})$, is bounded. %(see Section \ref{Prelim} for more details). %This class of functions contains many standard functions, e.g.\ the functions $\lambda e^z$ for $\lambda \in \C\setminus\{0\}$, $\sin(z)$ and $\sin(z)/z$. Functions in $\mathcal{B}$ are said to be of \textit{bounded type}. For these functions the escaping set is always a subset of the Julia set \cite{EL92} and a main tool to study functions in class $\mathcal{B}$ is given by the logarithmic change of variable (see Section \ref{Prelim} for more details).
%Hairs appear for $\lambda e^z$ for all $\lambda \in \C\setminus\{0\}$ (see \cite{DGH86,SZ03}) and also for more general classes of functions (see \cite{Bar07,DT86,RRRS11}). 

The next result is a generalization of Theorem \ref{ThmA}. We only state the part of the result by Bara\'nski \cite{Bar07} and by Rottenfu{\ss}er, R\"uckert, Rempe and Schleicher \cite{RRRS11}, which is relevant for us.
\begin{thmx}\label{RRRS}
	Let $f\in\mathcal{B}$ be of finite order. Then for $\lambda \in \R$ chosen small enough, the map $\lambda f$ has a unique attracting fixed point $\xi\in \C$ and $\C\setminus A(\xi)$ is an uncountable union of pairwise disjoint hairs.
\end{thmx}

Roughly speaking, an entire function is said to be of finite order, if it does not grow too fast (see Section \ref{Prelim} for a precise definition).
	
In \cite{Co17} it was shown that Viana's result generalizes to Zorich maps, which are a higher dimensional counterpart of the exponential family. 
 
The main aim of this article is to show that this is not the case for general functions in class $\mathcal{B}$. So even for analytic functions the hairs need not be smooth. More precisely we give an example of a function in class $\mathcal{B}$ of finite order with hairs which are nowhere differentiable. 

\begin{thm}\label{THM2}
	There exists a function $f\in \B$ as in Theorem \ref{RRRS}, where the set $\C\setminus A(\xi)$ consists of hairs which are nowhere differentiable.
\end{thm}
By saying that a hair $H\subset \C$ is nowhere differentiable, we mean that no parametrization $\gamma\colon [0,\infty) \to H$ has a non-zero derivative at any point of $[0,\infty)$.

To obtain this result we use the so-called Cauchy integral method to construct a transcendental entire map of finite order which grows in a prescribed tract to $\infty$. The existence of hairs for such a function is guaranteed by Theorem \ref{RRRS}. By choosing the geometry of the tract in a suitable way, one can obtain that the hairs cannot be differentiable at any point.

We conclude this introduction with a number of remarks.

\begin{remark}
	In complex dynamics the set $\C\setminus A(\xi)$ considered in Theorem \ref{ThmA} and \ref{RRRS} is called the \textit{Julia set} of $f$, denoted by $J(f)$. It is the set of all points in $\C$ where the iterates $f^k$ of $f$ do not form a normal family in the sense of Montel. Moreover, it turns out that the hairs except for their endpoints lie in the set $I(f):=\{z : \lim_{k\to\infty} f^k(z)=\infty\}$, which we call the \textit{escaping set} of $f$.
	
	For further information on complex dynamics we refer to \cite{Bea91,Ber93,Mi06,St93}.
\end{remark}

\begin{remark} 
		The conclusion of the results stated in Theorem \ref{ThmD} and \ref{RRRS} remains true in a more general setting. Viana \cite{Vi88} showed that the hairs of $\lambda e^z$ are smooth for all $\lambda \in \C\setminus\{0\}$. In \cite{RRRS11} it was shown that hairs exists for all functions in class $\mathcal{B}$ of finite order. In the same paper it was shown that the assumption on the order of $f$ cannot be dropped. In fact, they constructed a function of bounded type of infinite order such that every path-connected component of $J(f)$ is bounded and thus this function has no hairs. This answered a question of Eremenko to the negative whether for every $z\in I(f)$ there exists an unbounded and connected set $A\subset \C$ with $z\in A$ such that $f^n|_A \to \infty$ uniformly.
\end{remark}

\begin{remark}
	Beginning with Weierstra{\ss} in 1895 \cite{We95}, it is well-known that there exist continuous functions $f\colon \R\to\R$ which are nowhere differentiable. Fatou \cite{Fat06} has already shown that for $f(z)=z^2+c$, the Julia set $\J(f)$ is a Jordan curve for sufficiently small $\abs{c}\in\R$, but it is nowhere differentiable.
	Theorem \ref{THM2} shows that the pathological example of Weierstra{\ss} \cite{We95} has a dynamical counterpart in terms of hairs of entire functions.
\end{remark}

\begin{ack}
	I would like to thank Walter Bergweiler, Lasse Rempe-Gillen and Dan Nicks for valuable suggestions.
\end{ack}

	\section{Preliminaries}\label{Prelim}
	In this section we collect geometrical properties of curves and the main tools from complex dynamics. 
	  \begin{defn}[Order of a function] \label{finite}
	  	We say that an entire function $f$ is of \textit{finite order} if there exist $\mu,R>0$ such that 
	  	\begin{equation}
	  	\begin{aligned}
	  	\abs{f(z)}\leq \exp\left(\abs{z}^\mu\right) \quad \text{ for } \abs{z}\geq R.
	  	\end{aligned}
	  	\end{equation}
	  	The infimum over all $\mu$ such that this holds for some $R>0$ is called \textit{the order} of $f$ and is denoted by $\rho(f)$, where we use the convention $\inf \emptyset = \infty$, if such an $R$ does not exist for any $\mu >0$. Then we say that $f$ is of \textit{infinite order}.
	  	\end{defn}
	  		The following result \cite[Chapter 5, Theorem 1.2]{GO08} gives a connection between the number of asymptotic values and the order of an entire function.
	  			\begin{lem}[Denjoy-Carleman-Ahlfors] \label{DCA thm}
	  				Let $f$ be entire. Then the number of asymptotic values of $f$ is at most $\max\{1,2\rho(f)\}$.
	  				\end{lem}
	  				
	  				The next result can be found in \cite[p.\ 46]{Du04}.
	  				\begin{lem}\label{Radius of convexity}
	  					Let $f\colon \D \to \C$ be univalent. Then $f(D(0,r))$ is a convex set for $r\leq \sqrt{2}-1$. 	
	  					\end{lem}
	  				%	\begin{remark}
	  					%	The largest $r$ such that $f(D(0,r))$ is convex is called the \textit{radius of convexity}. Hengartner and Schober \cite{HS73} and Goodman and Saff \cite{GS79} have shown that $f(D(0,r))$ is for $r >\sqrt{2}-1$ in general not convex. Goodman and Saff conjectured that $r_0:=\sqrt{2}-1$ is the best possible constant to preserve convexity. In 1989 Ruscheweyh and Salinas \cite{RS89} proved that this constant is optimal. 
	  					%	\end{remark}
	  		\begin{defn}[Distortion]
				Let $U\subset \C$ be open and $f\colon U \to \C$ univalent. If $D\subset U$ is bounded, then we call
	 			\begin{equation}
	  			\begin{aligned}
					\mathcal{L}(f|_D):= \dfrac{\sup \{\abs{f'(z)} : z\in D\}}{\inf \{\abs{f'(z)} : z\in D\}}
	  			\end{aligned}
	 			\end{equation}
	  			the \textit{distortion of f on D}.
	  		\end{defn}
	  		The following Lemma can be found in \cite[Theorem 1.3]{Po92}.
			\begin{lem}[Koebe's distortion theorem]
 				Let $r>0$, $z_0\in\C$ and let $f\colon D(z_0,r)\to \C$ be univalent. Then we have for all $z\in D(z_0,r)$ and $s:=\abs{z-z_0}$
						\begin{equation}
	  					\begin{aligned}
	  					\dfrac{r-s}{(r+s)^3}\abs{f'(z_0)} \leq \abs{f'(z)} \leq \dfrac{r+s}{(r-s)^3}\abs{f'(z_0)}.
	  					\end{aligned}
	  					\end{equation}
	  		\end{lem}
	  		\begin{remark} \label{Koebe_Verzerrung}
	  			Under the assumptions of Koebe's distortion theorem we have
	  				\begin{equation}
	  				\begin{aligned}
	  				\mathcal{L}\left(f|_{D(z_0,s)}\right) \leq \left(\dfrac{r+s}{r-s}\right)^4.
	  				\end{aligned}
	  				\end{equation}
			\end{remark}
	  		In the following denote by $\Delta(z_1,z_2,z_3)\subset \C$ the triangle with vertices $z_1$, $z_2$ and $z_3$, where these points are distinct. Moreover, put $d_{j,k} := \abs{z_j-z_k}$ for $j,k\in\{1,2,3\}$ with $j\neq k$. 
	  			\begin{defn}[Degenerate/$\delta$-degenerate] \label{degen}
	  			We say that for $\delta\in (0,1)$ the triangle $\Delta(z_1,z_2,z_3)$ is \textit{$\delta$-close to degenerate}, if there exists different $j,k,l\in\{1,2,3\}$ such that
	  		\begin{equation}
			\begin{aligned}
 			\dfrac{d_{j,k}}{d_{j,l}+d_{k,l}}\geq 1-\delta. \label{eq:dreiecksungl}
	  		\end{aligned}
	 		\end{equation}
			If the quotient equals 1, we say that $\Delta(z_1,z_2,z_3)$ is \textit{degenerate}.
	 		\end{defn}

	  				By the triangle inequality we know that the quotient in \protect \eqref{eq:dreiecksungl} is always less than or equal to 1. Thus equality means that the points $z_1$, $z_2$ and $z_3$ lie on a line. In other words, if we have 
	  					\begin{equation}
	  					\begin{aligned}
	  					\mathcal{D}\left(\Delta(z_1,z_2,z_3)\right):=\max\left\{\dfrac{d_{j,k}}{d_{j,l}+d_{k,l}} : j,k,l\in\{1,2,3\} \text{ different} \right\} < 1-\delta, \label{eq:gutes_Dreieck} 
	 					\end{aligned}
	  					\end{equation}
	 						then the considered triangle is non-degenerate.
	  				
	   \begin{lem}[Distortion] \label{Kleine_Verzerrung}
	   	Let $U\subset \C$ be an open, bounded and convex set.  \\ Let $\Delta(z_1,z_2,z_3)\subset U$ and let $f\colon U \to \C$ be univalent. Suppose that $f(U)$ is convex. Then
	   	\begin{equation}
	   	\begin{aligned}
	   	\mathcal{D}\left(\Delta(f(z_1),f(z_2),f(z_3))\right) \leq \mathcal{L}\left(f|_{U}\right)\cdot\mathcal{D}\left(\Delta(z_1,z_2,z_3)\right). \label{eq:Verzerrung_Dreieck}
	   	\end{aligned}
	   	\end{equation}
	   \end{lem}
	   \begin{proof}
	   	We have for all $j,k\in\{1,2,3\}$ with $j\neq k$ that
	   	\begin{equation}
	   	\begin{aligned}
	   	\abs{f(z_j)-f(z_k)} &= \abs{\int_{z_k}^{z_j} f'(z)\, dz}  \\
	   	&\leq d_{j,k} \cdot \sup\{\abs{f'(z)} : z\in [z_j,z_k] \} \\
	   	&\leq d_{j,k}\cdot \sup\{\abs{f'(z)} : z\in U \}.
	   	\end{aligned}
	   	\end{equation}
	   	Since $f$ is univalent in $U$, we have
	   	\begin{equation}
	   	\begin{aligned}
	   	d_{j,k}  &= \abs{\int_{f(z_k)}^{f(z_j)} \left(f^{-1}\right)'(z)\, dz}  \\
	   	&\leq  \abs{f(z_j)-f(z_k)} \cdot\sup\left\{\abs{\left(f^{-1}\right)'(z)} : z\in [f(z_j),f(z_k)] \right\} \\
	   %	& = \abs{f(z_j)-f(z_k)} \cdot \sup\left\{\dfrac{1}{\abs{f'(f^{-1}(z))}} : z\in [f(z_j),f(z_k)] \right\} \\
	   	& \leq \dfrac{\abs{f(z_j)-f(z_k)}}{\inf\{\abs{f'(z)} : z\in U \}}.
	   	\end{aligned}
	   	\end{equation}
	   %	and hence
	   %	\begin{equation}
	   %	\begin{aligned}
	   %	\abs{f(z_j)-f(z_k)} &\geq d_{j,k} \cdot \inf\left\{\abs{f'(f^{-1}(z))} : z\in [f(z_j),f(z_k)] \right\} \\
	   %	& \geq d_{j,k} \cdot \inf\{\abs{f'(z)} : z\in U \}.
	   %	\end{aligned}
	   %	\end{equation}
	   	Thus we obtain for all pairwise different $j,k,l\in\{1,2,3\}$
	   	\begin{equation}
	   	\begin{aligned}
	   	\dfrac{\abs{f(z_j)-f(z_k)}}{\abs{f(z_j)-f(z_l)}+\abs{f(z_k)-f(z_l)}} \leq \dfrac{d_{j,k}}{d_{j,l}+d_{k,l}} \cdot \dfrac{\sup\{\abs{f'(z)} : z\in U \}}{\inf\{\abs{f'(z)} : z\in U\}}.
	   	\end{aligned}
	   	\end{equation}
	   	This yields
	   	\begin{equation}
	   	\begin{aligned}
	   	\mathcal{D}\left(\Delta(f(z_1),f(z_2),f(z_3))\right) \leq \mathcal{L}\left(f|_{U}\right)\cdot \mathcal{D}\left(\Delta(z_1,z_2,z_3)\right).
	   	\qedhere  
	   	\end{aligned}
	   	\end{equation}
	   \end{proof}
	   
	   The next lemma follows by elementary geometry.
	   \begin{lem} \label{winkel_weg_von_0}
	   	Let $\Delta(z_1,z_2,z_3)\subset \C$ be non-degenerate in the sense of Definition \ref{degen} for some $\delta>0$. Then every angle in $\Delta(z_1,z_2,z_3)$ is bounded away from $0$ and $\pi$, with bounds depending only on $\delta$.
	   \end{lem}
%	   \begin{proof}
%	   	We show the contraposition of the result. Without loss of generality we assume that $d_{1,2}$ is the longest edge of $\Delta(z_1,z_2,z_3)$. If $\alpha_j$ denotes the angle at the vertex $z_j$, the law of sines yields that the angle $\alpha_3$ is the largest one. We consider two cases: 
%	   	\medskip \\
%	   	\textbf{1. Both angles $\alpha_1$ and $\alpha_2$ are small.} 
%	   	
%	   	Suppose that $\delta'>0$ is small and that $\alpha_1,\alpha_2 < \delta'$. Put $\delta:=1-\cos\delta' \in (0,1)$. Then $\delta$ is close to 0. By elementary geometry we know that 
%	   	\begin{equation}
%	   	\begin{aligned}
%	   	d_{1,2} = d_{1,3}\cos \alpha_1+d_{2,3}\cos \alpha_2 \geq (d_{1,3}+d_{2,3})\cos \delta'
%	   	\end{aligned}
%	   	\end{equation}
%	   	and thus
%	   	\begin{equation}
%	   	\begin{aligned}
%	   	\dfrac{d_{1,2}}{d_{1,3}+d_{2,3}} \geq \cos\delta' = 1-\delta.
%	   	\end{aligned}
%	   	\end{equation}
%	   	Thus $\Delta(z_1,z_2,z_3)$ is $\delta$-close to degenerate if $\alpha_3$ is close to $\pi$. 
%	   	\medskip \\
%	   	\textbf{2. Only one of the angles $\alpha_1$ or $\alpha_2$ is small.} 
%	   	
%	   	Suppose now that $\delta'>0$ is small and assume without loss of generality that $\alpha_1 < \delta'$. Thus we know that $d_{2,3}$ is small compared to the length of the other edges. Using similar arguments as in the first case, we obtain the result.
%	   \end{proof}
	   
	   \begin{lem} \label{gamma_Dreieck}
	   	Let $\gamma\colon [a,b]\to\C$ be continuous and injective, and differentiable at $t_0\in[a,b]$ with $\gamma'(t_0)\neq 0$. Then for all $\delta \in (0,1)$ there exists $\varepsilon>0$ such that the triangle $\Delta(z_1,\gamma(t_0),z_2)$ is $\delta$-close to degenerate for all $z_1,z_2\in D(\gamma(t_0),\varepsilon)\cap \gamma([a,b])$.
	   \end{lem}
	   \begin{proof}
	   	% Since $\gamma$ is differentiable in $c$, we know that $\arg \gamma'(c)$ exists.
	   	Since $\gamma$ is injective, we know that for all $\delta >0$ there exists $\varepsilon>0$ such that for all $t\in [a,b]$ with $\abs{\gamma(t)-\gamma(t_0)}<\varepsilon$ we have $\abs{t-t_0}<\delta$. For all $t\in [a,b]\setminus\{t_0\}$ we have
	   	\begin{equation}
	   	\begin{aligned}
	   	\dfrac{\gamma(t)-\gamma(t_0)}{t-t_0} = \gamma'(t_0) + \kappa(t)
	   	\end{aligned}
	   	\end{equation}
	   	with $\kappa(t)\to 0$ as $t\to t_0$. Let $\eta >0$ be small. Then  we can choose $\delta$ small enough such that for all $t\in [a,b]$ with $\abs{t-t_0}<\delta$ we have $\abs{\kappa(t)}< \abs{\gamma'(t_0)}\sin \eta$.
	   	
	   	Let $z_1,z_2\in \left(D(\gamma(t_0),\varepsilon)\cap \gamma([a,b])\right)\setminus\{\gamma(t_0)\}$, say $z_1=\gamma(t_1)$ and $z_2=\gamma(t_2)$. Then $\abs{t_j-t_0}<\delta$ and we obtain
	   	\begin{equation}
	   	\begin{aligned}
	   	\abs{\operatorname{arg}\left(\gamma'(t_0)+\kappa(t_j)\right)-\operatorname{arg}\left(\gamma'(t_0)\right)}\leq \arcsin\left(\dfrac{\abs{\gamma'(t_0)}\sin\eta }{\abs{\gamma'(t_0)}}\right) = \eta
	   	\end{aligned}
	   	\end{equation}
	   	for $j\in\{1,2\}$. Thus the equation
	   	\begin{equation}
	   	\begin{aligned}
	   	\gamma(t_j)-\gamma(t_0) = \left(\gamma'(t_0) + \kappa(t_j)\right)(t_j-t_0)
	   	\end{aligned}
	   	\end{equation}
	   	yields 
	   	\begin{equation}
	   	\begin{aligned}
	   	\operatorname{arg}\left(\gamma(t_j)-\gamma(t_0)\right) = \operatorname{arg} \gamma'(t_0) +r_j,
	   	\end{aligned}
	   	\end{equation}
	   	where $\abs{r_j}\leq \eta$ for $t_j >t_0$ and $\abs{r_j-\pi}\leq \eta$ for $t_j < t_0$ for a suitable choice of the argument. Hence
	   	\begin{equation}
	   	\begin{aligned}
	   	\abs{\operatorname{arg}\left(\gamma(t_2)-\gamma(t_0)\right)-\operatorname{arg}\left(\gamma(t_1)-\gamma(t_0)\right)} \leq 2\eta
	   	\end{aligned}
	   	\end{equation}
	   	or
	   	\begin{equation}
	   	\begin{aligned}
	   	\abs{\operatorname{arg}\left(\gamma(t_2)-\gamma(t_0)\right)-\operatorname{arg}\left(\gamma(t_1)-\gamma(t_0)\right)-\pi} \leq 2\eta.
	   	\end{aligned}
	   	\end{equation}
	   	Noticing that $\operatorname{arg}\left(\gamma(t_2)-\gamma(t_0)\right)-\operatorname{arg}\left(\gamma(t_1)-\gamma(t_0)\right)$ is the angle of the triangle at $\gamma(t_0)$, the conclusion follows.
	   \end{proof}
	
	Now we introduce a key tool given by \cite{EL92}, the so-called logarithmic change of variable, to study functions in class $\B$. We have the following result.
	
	\begin{thm}[Logarithmic change of variable]
		Let $f\in\B$ and let $R>\abs{f(0)}$ be such that $\sing(f^{-1})\subset D(0,R)$. Let 
		\begin{equation}
		\begin{aligned}
		W:=\exp^{-1}\left(f^{-1}\left(\left\{z\in\C : \abs{z}>R\right\}\right)\right)
		\end{aligned}
		\end{equation}
		and
		\begin{equation}
		\begin{aligned}
		\mathbb{H}:= \{z\in\C : \Re z > \log R\}.
		\end{aligned}
		\end{equation}
		Then every component of $W$ is simply connected and there exists a holomorphic function $F\colon W\to \mathbb{H}$ such that
		\begin{equation}
		\begin{aligned}
		\exp F(z) = f(\exp z) 
		\end{aligned}
		\end{equation}
		for all $z\in W$. Moreover, for every component $U$ of $W$ the restriction $F|_U$ is a conformal map from $U$ onto $\mathbb{H}$.
	\end{thm}
	The function $F$ obtained by the logarithmic change of variable of $f$ is called a \textit{logarithmic transform of }$f$. Moreover, we call every component of $W$ a \textit{tract}.
	
	 There are several ways to define disjoint type functions. \cite[Lemma 3.1]{Bar07} or \cite[Proposition 2.8]{MB12} provide the definition of disjoint type entire functions which we will use in the sequel.
	
	\begin{defn}[Disjoint type functions] \label{disjoint type}
		A transcendental entire function $f$ is said to be of \textit{disjoint type} if there exists a bounded Jordan domain $D\subset \C$ such that $\sing(f^{-1})\subset D$ and $f\left(\overline{D}\right) \subset D$.
	\end{defn}
	
		\begin{remark}\label{f_0=lambda f}
			Since the set $\sing(f^{-1})$ is bounded for every function $f\in\B$, one can obtain a function of disjoint type by considering the function $f_0:=\lambda f$ for small $\lambda \in \R$. This is precisely the situation which was considered in Theorem \ref{RRRS}.
			
			 If in addition $\lambda$ is chosen small enough, then there exists a constant $K > 1$ such that $\abs{F'(z)}\geq K$ for all $z\in W$ and $\sing(f_0^{-1})\subset \D$. Moreover, we have $D\subset \F(f_0)$ by Montel's theorem.
		\end{remark}
		
		Following \cite[Section 3]{BJR12}, we can define disjoint type logarithmic transforms as follows.
		
	\begin{defn}[Disjoint type for $F$]
		Let $F \colon W \to \mathbb{H}$ be a logarithmic transform of the function $f\in\B$. Then we say that $F$ is of \textit{disjoint type}, if $\overline{W}\subset \mathbb{H}$.
	\end{defn}
	
	We only state the part of \cite[Lemma 3.1]{BJR12} which is relevant for us.
	
	\begin{lem}[Connection between disjoint type functions]
		Let $F\colon W\to \mathbb{H}$ be a logarithmic transform of the function $f\in\B$. Then $f$ is of disjoint type if and only if $F$ is of disjoint type.
	\end{lem}
	
	\begin{remark} \label{Baranski2}
			It was shown in \cite[Proposition 4.2, Corollary 4.4]{BJR12} that the Julia set $\J(F)$ of a logarithmic transform $F\colon W\to \mathbb{H}$ of a disjoint type function $f$ of finite order also consists of uncountable union of hairs which are pairwise disjoint, where
			\begin{equation}
			\begin{aligned}
			\J(F):=\{z\in \overline{W} : F^j(z)\in\overline{W} \text{ for all } j\in\N_0\}.
			\end{aligned}
			\end{equation} 
			Note that this continuous extension of $F|_W$ exists by considering a smaller value of $R$.
			
			Moreover, we have $\exp(\J(F))=\J(f)$. This yields that, given a parametrization $\gamma$ of a hair $H\subset \J(F)$, the function $\exp \circ \gamma$ is a parametrization of the hair $\exp(H)$ in $\J(f)$. Thus $H$ is nowhere differentiable if and only if $\exp(H)$ is nowhere differentiable.
	\end{remark}
	
	\begin{remark}\label{Tracts_log_coord}
	Now we want to discuss the existence of hairs.
	Given a function $f\in \mathcal{B}$ such that the set $T:=\Log\left(\left(f^{-1}\left(\left\{z\in\C : \abs{z}>R\right\}\right)\right)\right)$ consists of exactly one component, where $\Log$ denotes the principle branch of the logarithm, we put $T(0):=T$ and for $k\in\Z$ denote $T(k):=T(0)+2\pi ik$. With
	\begin{equation}
	\begin{aligned}
	\mathcal{T}:= \bigcup_{k\in\Z} T(k)
	\end{aligned}
	\end{equation}
	the logarithmic transform of $f$ takes the form $F\colon \mathcal{T}\to \mathbb{H}$.
	\end{remark}
	\begin{defn}[External address]
		For each $z\in J(F)$ we call the sequence
		\begin{equation}
		\ul{s}(z):=s_0 s_1 s_2 \dots = (s_k)_{k\geq 0} \in \Z^{\N_0}    
		\end{equation}
		such that $F^k(z)\in T(s_k)$ for all $k\geq 0$ the \textit{external address of $z$}.
	\end{defn}
	
	For disjoint type maps of finite order we will use the following Theorem \cite[Theorem 4.7]{RRRS11}.
	\begin{thm}\label{Hairs}
		Let $\ul{s}\in \Z^{\N_0}$. Then $\{z\in J(F) : \ul{s}(z)=\ul{s}\}$ is either empty or a hair. 
	\end{thm}
	
	\begin{remark}\label{Re F}
	For a logarithmic transform of $F$ of a disjoint type map $f$ of finite order the hairs of $F$ except for the endpoints belong to the escaping set of $F$, i.e.\ given a parametrization $\gamma\colon [0,\infty)\to\C$ of a hair of $F$, we have $\Re F^n|_{\gamma((0,\infty))} \to \infty$ locally uniformly as $n\to\infty$.
	\end{remark}
	
	\section{Construction of the functions and the Cauchy integral method}
	In this section we use for $\alpha\in\R$ and $\beta >0$ the notation
	\begin{equation}
	\begin{aligned}
	S_{\alpha,\beta}:=\{z\in\C : \Re(z)>\alpha, \abs{\Im(z)}< \beta\}. 
	\end{aligned}
	\end{equation}
	The original Cauchy integral method was established by P\'olya and Szeg\"o \cite[p.\ 115, ex.\ 158]{PS72} using the function 
	\begin{equation}
	\begin{aligned}	
	\tilde{f}(z)= \dfrac{1}{2\pi i}\int_{\gamma} \dfrac{e^{e^\zeta}}{\zeta -z} d\zeta,
	\end{aligned}
	\end{equation}
	where $\gamma$ is a curve going through the boundary of $S_{0,\pi}=\{z\in\C : \Re(z)>0, \abs{\Im(z)} < \pi \}$ in clockwise direction and $z \in \C\setminus \overline{S_{0,\pi}}$. This function is bounded in $\C\setminus\overline{S_{0,\pi}}$. By considering the integral over $\partial S_{\alpha,\pi}$ with $\alpha>0$ instead of $\gamma$ and letting $\alpha\to\infty$, one can extend the function $\tilde{f}$ to an entire function $f$ in such a way that
	\begin{equation}
	\begin{aligned}
	f(z)=
	\begin{dcases}
	\tilde{f}(z) & \text{, if } z\in \C\setminus \overline{S_{0,\pi}},\\
	\tilde{f}(z)+e^{e^z} & \text{, if } z \in S_{0,\pi}. \\
	\end{dcases}
	\end{aligned}
	\end{equation}
	
	For a thorough discussion on the Cauchy integral method we refer to \cite[Theorem 1.7]{RG14}.
	
	In the following we consider a tract which looks like a sine-like halfstrip in logarithmic coordinates. Our aim is to apply the Cauchy integral method in `standard coordinates' to obtain a suitable function of bounded type. 
	
	To realize this tract, we define the function
	\begin{equation}
	\begin{aligned}
	h \colon \C \to \C,\quad h(z)=5\pi z+2\pi i\sin z. 
	\end{aligned}
	\end{equation}
	Then the function $h|_{S_{-1,\pi/3}}$ is univalent since $\Re h'(z)>0$ and this implies that
	\begin{equation}
	S_{0,\pi} \subset h(S_{-1,\pi/3}).
	\end{equation}
	
	 \begin{remark} \label{tract_constants}
		For our purpose it is useful to show that $T:=h^{-1}(S_{0,\pi})$ contains no straight line which is unbounded to the right. 
		
		Since 
		\begin{equation}
		\begin{aligned}
		\Im h(x+iy) = 5\pi y + 2\pi \sin x \cdot \cosh y,
		\end{aligned}
		\end{equation}
		we first consider the equation $\Im h(x+iy)=\pi$ for the upper boundary of $T$. Thus we have for all $k\in\Z$ and $x=x_k=\pi/2 + 2\pi k$
		\begin{equation} \label{eq:pi}
		\begin{aligned}
		5\pi y + 2\pi \cosh y = \pi
		\end{aligned}
		\end{equation}
		and since $\cosh y \geq 1$, we obtain $y\leq -1/5$. For the lower boundary of $T$ we consider the equation $\Im h(x+iy)=-\pi$ and with $x=x_k=-\pi/2 + 2\pi k$ we obtain
		\begin{equation} \label{eq:-pi}
		\begin{aligned}
		5\pi y - 2\pi \cosh y = -\pi
		\end{aligned}
		\end{equation}
		and hence $y\geq 1/5$. This implies that $T$ contains no straight line.
		
		For the proof of Theorem \ref{THM2} we need a bit more information about the geometry of $T$.
		If $y=0$, then the equation $\Im h(x+iy)=\pi$ yields $2\pi \sin x = \pi$ and thus we have $\sin x = 1/2$. Hence the solutions in $[0,2\pi]$ are $\pi/3$ and $2\pi/3$. Similarly, we obtain for $y=0$ in $\Im h(x+iy)=-\pi$ the solutions $4\pi/3$ and $5\pi/3$.

	Note that $T$ is $2\pi$-periodic, i.e.\ if $z\in T$, then $z+2\pi \in T$. Using the fact that $\abs{y}\leq \pi/3$ and equation $\Im h(x+iy)=\pi$ with $x=x_k=-\pi/2 + 2\pi k$ resp. $\Im h(x+iy)=-\pi$ and $x=x_k=\pi/2+2\pi k$, we similarly obtain that $\abs{\Im(z)}\leq 7/8$ for all $z\in T$. More precisely, one can show that $\abs{\Im(z)}\leq 3/4$ for all $z\in T$.
	 \end{remark}
	\begin{figure}[H]
		\begin{center}
			\includegraphics[height=8.5cm]{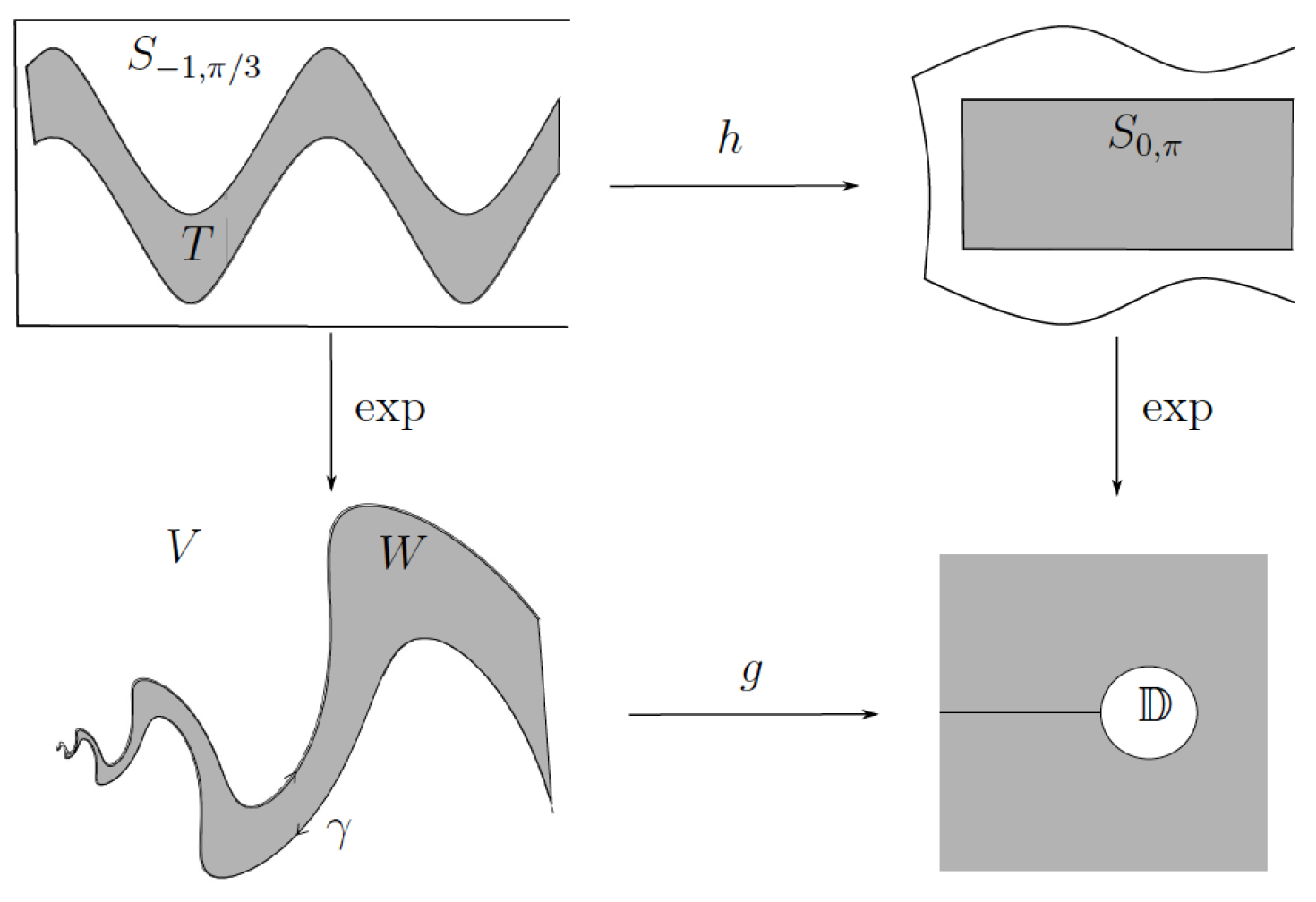}
			\caption{Construction of the function $g$ using the logarithmic change of variable of the map $h$.}
		\end{center}
	\end{figure}
	
	We define
	\begin{equation}
	g \colon \left\{z\in\C : \abs{z}>1/e , \abs{\arg z}<\pi/3\right\} \to \C,\, g(z)=\exp\left(h(\Log z)\right),
	\end{equation}
	where $\Log$ denotes the principal branch of the complex logarithm. Then $g$ is holomorphic and with $W:=\exp(T)$ we obtain 
	\begin{equation}
	\begin{aligned}
	\overline{W}\subset \{z\in\C : \abs{z}>1/e , \abs{\arg z}<\pi/3\}.
	\end{aligned}
	\end{equation} 
	Now we parametrize $\partial W$ by its radius for large $R>1$, i.e.\ we define
	\begin{equation}
	\gamma \colon \R \to \partial W, \quad \gamma(t)=
	\begin{dcases}
	\abs{t}e^{i\varphi(t)}, & \text{if } \abs{t}\geq R, \\
	\tilde{\gamma}(t), & \text {else},
	\end{dcases}
	\end{equation}
	where $\varphi$ is a function and  $\tilde{\gamma}$ is a suitable piecewise $C^1$ curve `gluing' the outer parts together. Because $\Re h'(z) >0$ we have $\abs{\arg h'(z)}<\pi/2$ and the upper respectively the lower boundary of $T$ has a parametrization $t\mapsto \Log \abs{t}+i\varphi(t)$. This guarantees the parametrization for $\partial W$ from above. 
	
	We want to show the following lemma which is similar to \cite[Prop. 7.1]{RRRS11}. Here $V:=\C\setminus W$.
	\begin{lem} \label{Konstruktion f}
		There exists an entire function $f\in \mathcal{B}$ of finite order and a constant $K>0$ such that
		\begin{enumerate}
			\item $A:=\{z\in\C : \abs{f(z)}>K\} \subset W.$ 
			\item We have $\abs{f(z)-e^{g(z)}} = O(1)$ on $A$ and $\abs{f(z)}=O(1)$ on $V.$
		\end{enumerate}
	\end{lem}
	We split the proof up into several parts. \bigskip \\
	\textbf{Claim 1.} The map
	\begin{equation}
	\tilde{f}(z):= \dfrac{1}{2\pi i}\int_{\gamma} \dfrac{\exp g(\zeta)}{\zeta -z} d\zeta
	\end{equation}
	defines a holomorphic function $\tilde{f}\colon \C\setminus \gamma \to \C$. 
	\begin{proof}
		We show that the integral
		\begin{equation}
		\dfrac{1}{2\pi i}\int_{\gamma} \exp g(\zeta) d\zeta
		\end{equation}
		converges absolutely for all $z\in\C\setminus\gamma$. For $\abs{t}\geq R$ we have
		\begin{equation}
		h(\Log(\gamma(t)))= 5\pi(\Log \abs{t} + i\varphi(t)) + 2\pi i \sin(\Log \abs{t} + i\varphi(t)).
		\end{equation}
		Since
		\begin{equation}
		\sin(\Log \abs{t} + i\varphi(t)) = \sin(\Log \abs{t})\cdot \cosh \varphi(t) + i \cos(\Log \abs{t})\cdot \sinh \varphi(t),
		\end{equation}
		we obtain
		\begin{equation} \label{eq:g_gamma}
		\begin{aligned}
		g(\gamma(t))  &=\exp\left(h(\Log(\gamma(t)))\right) \\
		&=\exp\left(5\pi \Log \abs{t}\right)\cdot \exp\left(5\pi i\varphi(t)\right)\cdot \exp\left(2\pi i\sin(\Log \abs{t} + i\varphi(t))\right) \\
		&=\abs{t}^{5\pi} \cdot\exp(-2\pi \cos(\Log \abs{t})\cdot \sinh \varphi(t)) \\ 
		&\hspace{0.5cm}\cdot \exp\left(\pi i\left(5 \varphi(t)+ 2\sin(\Log \abs{t})\cdot \cosh \varphi(t)\right)\right).				
		\end{aligned}
		\end{equation}
		By our construction of $h$ we know that for all $\abs{t}\geq R$
		\begin{equation}
		h(\Log \gamma(t)) \in \R^+ \pm i\pi
		\end{equation}
		and thus $g(\gamma(t)) \in (-\infty,0)$ for $\abs{t}\geq R$. Therefore we obtain
		\begin{equation}
		5\varphi(t)+2\sin(\Log \abs{t})\cdot \cosh \varphi(t) =
		\begin{dcases}
		1, & \text{if } t\geq R, \\
		-1, & \text{if } t \leq -R.\\
		\end{dcases}
		\end{equation}
		Differentiating both sides with respect to $t$ we obtain
		\begin{equation}
		5\varphi'(t) + 2\dfrac{\cos(\Log \abs{t})}{t}\cdot \cosh \varphi(t) + 2\sin(\Log \abs{t})\cdot \sinh \varphi(t) \cdot \varphi'(t) = 0.
		\end{equation}
		From the construction of $h$ it is clear that $\abs{\varphi(t)}\leq \pi/3$ for all $t$. Since
		\begin{equation}
		5 + 2\sin(\Log \abs{t})\cdot \sinh \varphi(t) \geq  5-2\sinh \frac{\pi}{3} > \frac{5}{2},
		\end{equation}
		this yields
		\begin{equation}
		\abs{\varphi'(t)} = \frac{1}{\abs{t}} \cdot \dfrac{2\abs{\cos(\Log \abs{t})}\cdot \cosh \varphi(t)}{\abs{5+2\sin(\Log \abs{t})\cdot \sinh \varphi(t)}} \leq \dfrac{4\cosh \frac{\pi}{3}}{5\abs{t}} \leq \dfrac{2}{\abs{t}}.
		\end{equation}
		With
		\begin{equation}
		\gamma'(t) = \operatorname{sgn}(t)\cdot e^{i\varphi(t)}\cdot (1+it\varphi'(t))
		\end{equation}
		we obtain
		\begin{equation} \label{eq:gamma'}
		\abs{\gamma'(t)} = \abs{1+it\varphi'(t)} \leq 3 
		\end{equation}
		for $\abs{t}\geq R$. Since $\gamma$ is piecewise smooth, there is a constant $C_1>0$ such that 
		\begin{equation} \label{eq:C_1}
		\int_{-R}^{R} \exp(\Re g(\gamma(t)))\cdot \abs{\gamma'(t)}dt \leq C_1.
		\end{equation}
		With $C_2:=\exp\left(-2\pi \sinh \frac{\pi}{3}\right)$ we have
		\begin{equation}
		\begin{aligned} \label{eq:C_2}
		\exp\Re g(\gamma(t)) &= \exp\left(-\abs{t}^{5\pi}\cdot \exp\left(-2\pi \cos(\Log \abs{t})\cdot \sinh \varphi(t)\right)\right) \\
		&\leq \exp\left(-\abs{t}\cdot \exp\left(-2\pi \sinh \frac{\pi}{3}\right)\right) \\
		&= \exp\left(-C_2\cdot\abs{t}\right).
		\end{aligned}
		\end{equation}
		It follows from \protect \eqref{eq:gamma'}, \eqref{eq:C_1} and \eqref{eq:C_2} for $z\in \C\setminus\gamma$ that
		\begin{equation}  \label{eq:f_schlange1}
		\begin{aligned} 
		\abs{\dfrac{1}{2\pi i}\int_{\gamma} \exp g(\zeta) d\zeta}
		%&\leq\dfrac{1}{2\pi}\int_{\gamma} \exp{\Re g(\zeta)}\abs{d\zeta}\\
		\leq \dfrac{1}{2\pi} \int_{-\infty}^{\infty} \exp{\Re g(\gamma(t))}\cdot \abs{\gamma'(t)} dt 
	%	&\leq \dfrac{1}{2\pi} \left(C_1+\int_{-\infty}^{-R} \exp(\Re g(\gamma(t))) \abs{\gamma'(t)}dt + \int_{R}^{\infty} \exp(\Re g(\gamma(t))) \abs{\gamma'(t)}dt \right) \\
		%\leq \dfrac{1}{2\pi} \left(C_1+6\int_{R}^\infty \exp\left(-C_2 t\right) dt  \right) 
		\leq \dfrac{1}{2\pi}  \left(C_1+\dfrac{6}{C_2}\right). 
		\end{aligned}
		\end{equation}
		Thus $\int_{\gamma} \exp(g(\zeta)) d\zeta$ converges absolutely and the claim follows.
	\end{proof}
	\noindent
	\textbf{Claim 2.} The function $\tilde{f}$ is bounded.
	\begin{proof}
		We choose $\kappa>0$ small such that the following holds:
		\begin{enumerate}   
			\item We have 
			\begin{equation}
			U_\kappa(\gamma) := \{w\in\C : \dist(\gamma,w)\leq \kappa\} \subset \{z\in\C : \abs{z}>1/e , \abs{\arg z}<\pi/3\}.
			\end{equation}
			\item For all $z\in U_\kappa(\gamma)$ we have 
			\begin{equation}
			\begin{aligned}
			g(z)\in \left\{3+re^{i\theta} : \theta \in (3\pi /4,5\pi/4)\right\}. \label{eq:kappa_2}
			\end{aligned}
			\end{equation}
		\end{enumerate}
		Moreover, we define
		\begin{equation}
		V_{\kappa}:=\{z\in \C\setminus \overline{W} : \operatorname{dist}(z,W)>\kappa\} 
		\end{equation}
		and
		\begin{equation}
		W_\kappa := \{z \in W : \dist(z,\partial W)> \kappa\}. 
		\end{equation}
		The first step is to show that $\tilde{f}$ is bounded in $V_{\kappa}\cup W_\kappa$. As long as $z\in V_\kappa \cup W_\kappa$ we use the curve $\gamma$ from above.
	%	\begin{figure}[H] 
			%%\includegraphics[viewport = 40mm 105mm 180mm 177mm, clip]{Masterpic.pdf_tex}
		%	\centering
		%	\def\svgwidth{350pt} \input{sinelike_kugel.pdf_tex} 
		%	\caption{.\, Construction of the curve $\gamma_z$ and its decomposition into smaller parts.} \label{gamma_z}
		%\end{figure}
		If $z\in U_\kappa(\gamma)$, we modify the curve as follows:
		
		Since $\gamma$ is unbounded, there exist $R_1 \in \R$ minimal and $R_2\in \R$ maximal such that $\dist(\gamma(R_j),z)=\kappa$ for $j\in\{1,2\}$. Moreover, there exist angles $\psi_1,\psi_2\in [0,2\pi)$ such that $\gamma(R_1)=R_1e^{i\varphi(R_1)}=z+\kappa e^{i\psi_1}$ and $\gamma(R_2)=R_2e^{i\varphi(R_2)}=z+\kappa e^{i\psi_2}$. Thus we define
		\begin{equation}
		\begin{aligned} \label{def_gamma_z}
		\gamma_z(t):=
		\begin{dcases}
		\gamma(t), & \text{if } t\in (-\infty,R_1)\cup(R_2,\infty),\\
		z+\kappa \exp\left(i\varphi_z(t)\right), & \text{if } t \in [R_1,R_2], \\
		\end{dcases}
		\end{aligned}
		\end{equation}
		with a suitable function $\varphi_z$ satisfying $\varphi_z(R_j)=\psi_j$ for $j\in\{1,2\}$.
		%Then $\gamma_z$ is piecewise $C^1$ and we obtain for $t\in(R_1,R_2)$
		%\begin{equation}
		%\begin{aligned}
		%\gamma_z'(t) = \pm\kappa i \dfrac{\psi_2-\psi_1}{R_2-R_1} \exp\left(\pm i\left(\psi_1+\dfrac{t-R_1}{R_2-R_1}(\psi_2-\psi_1)\right)\right)
		%\end{aligned}
		%\end{equation}
		%and thus it is not difficult to show that $\gamma_z'$ is bounded in terms of $\kappa$.
		
		The next step is to show that the integral
		\begin{equation} \label{integral_gamma_z}
		\begin{aligned}
		\dfrac{1}{2\pi i}\int_{\gamma_z} \dfrac{\exp g(\zeta)}{\zeta -z} d\zeta 
		\end{aligned}
		\end{equation}
		is bounded independently of $z$. Notice that by Cauchy's integral theorem we have 
		\begin{equation}
		\begin{aligned}
		f(z)= \dfrac{1}{2\pi i}\int_{\gamma} \dfrac{\exp g(\zeta)}{\zeta -z} d\zeta = \dfrac{1}{2\pi i}\int_{\gamma_z} \dfrac{\exp g(\zeta)}{\zeta -z} d\zeta.
		\end{aligned}
		\end{equation}
		Now we split $\gamma_z$ up into three parts. Therefore denote 
		\begin{equation}
		\begin{aligned}
		\gamma_{z,1}:=\gamma_z|_{(-\infty,R_1)}, \quad \gamma_{z,2}:=\gamma_z|_{[R_1,R_2]} \quad \text{and} \quad \gamma_{z,3}:=\gamma_z|_{(R_2,\infty)}.
		\end{aligned}
		\end{equation}
		Using this decomposition of $\gamma_z$, we obtain with \eqref{eq:f_schlange1}
		\begin{equation}
		\begin{aligned} \label{gamma_z1}
		\abs{\dfrac{1}{2\pi i}\int_{\gamma_{z}-\gamma_{z,2}} \dfrac{\exp g(\zeta)}{\zeta -z} d\zeta} 
		&\leq \dfrac{1}{2\pi} \int_{\gamma_{z}-\gamma_{z,2}} \dfrac{\abs{\exp{ g(\zeta)}}}{\abs{\zeta-z}}\abs{d\zeta} \\
		&\leq \dfrac{1}{2\pi \kappa} \int_{\gamma} \exp{\Re g(\zeta)}\abs{d\zeta} \\
		&\leq \dfrac{1}{2\pi \kappa}  \left(C_1+\dfrac{6}{C_2}\right). 
		\end{aligned}
		\end{equation}
		Moreover, we have with \protect \eqref{eq:kappa_2} and $\gamma_{z,2}\subset U_\kappa(\gamma)$
		\begin{equation}
		\begin{aligned} \label{gamma_z2}
		\abs{\dfrac{1}{2\pi i}\int_{\gamma_{z,2}} \dfrac{\exp g(\zeta)}{\zeta -z} d\zeta} &\leq \operatorname{length}(\gamma_{z,2}) \cdot \max\limits_{\zeta \in \gamma_{z,2}} \dfrac{\abs{\exp g(\zeta)}}{\abs{\zeta-z}} \\
		&\leq \dfrac{2\pi \kappa}{\kappa} \cdot \max\limits_{\zeta \in \gamma_{z,2}} \exp\left(\Re g(\zeta)\right) \\
		&\leq 2\pi e^3.
		\end{aligned}
		\end{equation}
		Thus the integral in \eqref{integral_gamma_z} is bounded independently of $z$ and hence there exists a constant $C_3>0$ such that
		\begin{equation} \label{eq:C_3}
		\begin{aligned}
		\abs{\tilde{f}(z)}\leq C_3
		\end{aligned}
		\end{equation}
		 for all $z\in\C\setminus\gamma$. \qedhere
		
	\end{proof} 
	\noindent
	\textbf{Claim 3.} The function
	\begin{equation}
	\begin{aligned}
	f(z):=
	\begin{dcases}
	\tilde{f}(z), & \text{if } z\in V,\\
	\tilde{f}(z)+\exp(g(z)), & \text{if } z \in W, \\
	\end{dcases}
	\end{aligned}
	\end{equation}
	extends to an entire function $f\colon \C \to \C$ of finite order.
	\begin{proof}
		Let $R_0 > 0$ be large and define
		\begin{equation}
		\begin{aligned}
		\gamma_1:= \left(\gamma \cap \{z\in\C : \abs{z}>R_0\}\right) \cup \left\{R_0e^{i\theta} : \theta\in [\varphi(-R_0),\varphi(R_0)]\right\}
		\end{aligned}
		\end{equation}
		and let $V'$ be the component of $\C\setminus \gamma_1$ containing $V$.
		
	%	\begin{figure}[H]
			%%\includegraphics[viewport = 40mm 105mm 180mm 177mm, clip]{Masterpic.pdf_tex}
		%	\centering
		%	\def\svgwidth{350pt} \input{sinelike_bea.pdf_tex}
		%	\caption{.\, Construction of the curve $\gamma_1$.}
		%\end{figure}
		Then the function
		\begin{equation}
		\begin{aligned}
		f_1 \colon V' \to \C, f_1(z):= \dfrac{1}{2\pi i}\int_{\gamma_1} \dfrac{\exp g(\zeta)}{\zeta -z} d\zeta
		\end{aligned}
		\end{equation}
		defines a holomorphic function on $V'$. By the Cauchy integral theorem we obtain for $z\in V$
		\begin{equation}
		\begin{aligned}
		f(z)-f_1(z) = \dfrac{1}{2\pi i}\int_{\gamma} \dfrac{\exp g(\zeta)}{\zeta -z} d\zeta - \dfrac{1}{2\pi i}\int_{\gamma_1} \dfrac{\exp g(\zeta)}{\zeta -z} d\zeta = 0.
		\end{aligned}
		\end{equation}
		Thus $f_1$ coincides with $f$ on $V$. Moreover, we have for $z\in V'\cap W$ by the Cauchy integral theorem that
		\begin{equation}
		\begin{aligned}
		f_1(z)-\tilde{f}(z) = \exp(g(z)).
		\end{aligned}
		\end{equation}
		Notice that $f_1 =f|_{V'}$. Since $R_0$ was arbitrary, the function $f$ is entire. It is left to show that $f$ is of finite order. With equality \eqref{eq:g_gamma} we obtain for $z=re^{i\theta}\in W$ 
		\begin{equation}
		\begin{aligned}
		\abs{g(re^{i\theta})} = r^{5\pi}\cdot \exp(-2\pi\cos\left(\Log r\right)\cdot \sinh \theta) 
		\leq r^{5\pi} \cdot \exp\left(2\pi\sinh \dfrac{\pi}{3}\right).
		\end{aligned}
		\end{equation}
		Since $f(z)=\exp(g(z))+O(1)$ for $z\in W$ and $f(z)=O(1)$ for $z\in\C\setminus W$, the function $f$ is of order at most $5\pi$ and thus the claim follows.
	\end{proof}
	\begin{proof}[Proof of Lemma \ref{Konstruktion f}.]
		It is left to show that $f\in\B$. By the theorem of differentiation under the integral sign, we have 
		\begin{equation}
		\begin{aligned}
		\tilde{f}'(z)= \dfrac{1}{2\pi i}\int_{\gamma} \dfrac{\exp g(\zeta)}{(\zeta -z)^2} d\zeta.
		\end{aligned}
		\end{equation}
		Because the only difference between $\tilde{f}$ and $\tilde{f}'$ is that $1/(\zeta-z)$ was replaced by $1/(\zeta-z)^2$ in the integral, this yields with \eqref{eq:f_schlange1} that
		\begin{equation}
		\begin{aligned}
		\abs{\tilde{f}'(z)} \leq \dfrac{1}{2\pi\kappa^2}\left(C_1+\dfrac{6}{C_2}\right)
		\end{aligned}
		\end{equation}
		for $z\in V_\kappa \cup W_\kappa$. For $z\in U_\kappa(\gamma)$ and $\gamma_z$ as in \eqref{def_gamma_z} we also have the same estimates in \eqref{gamma_z1} and \eqref{gamma_z2} with the extra factor $1/\kappa$ . This yields that we have in total $\abs{\tilde{f}'(z)}\leq C_3/\kappa$ on $\C\setminus\gamma$. 
		
		Now we show that $f\in\mathcal{B}$, i.e.\ that $\sing(f^{-1})$ is bounded. Since
		\begin{equation}
		\begin{aligned}
		f'(z)=
		\begin{dcases}
		\tilde{f}'(z), & \text{if } z\in V,\\
		\tilde{f}'(z)+g'(z)\exp(g(z)), & \text{if } z \in W, \\
		\end{dcases}
		\end{aligned}
		\end{equation}
		we obtain $\abs{f'(z)-g'(z)\exp(g(z))}\leq C_3/\kappa$ on $W$. Since $f$ is bounded on $V$, it suffices to show that the set of critical values of $f$ corresponding to critical points in $W$ is bounded. That is, the set $\{f(\zeta) : \zeta \in W, f'(\zeta)=0\}$ is bounded. Let $\xi\in W$ be a critical point of $f$. Thus 
		\begin{equation} \label{eq:f_in_B}
		\begin{aligned}
		\abs{f'(\xi)-g'(\xi)\exp(g(\xi))}=\abs{g'(\xi)}\cdot \exp(\Re(g(\xi)))\leq \dfrac{C_3}{\kappa}.
		\end{aligned}
		\end{equation}
		Moreover, we have
		\begin{equation}
		\begin{aligned}
		g'(z)&=\dfrac{\pi}{z}\left(5+2i\cos\left(\Log z\right)\right)\cdot \exp\left(5\pi\Log z + 2\pi i\sin\left(\Log z\right)  \right) \\
		&= \left(5\pi+2\pi i\cos\left(\Log z\right)\right)\cdot z^{5\pi -1} \cdot \exp\left(2\pi i \sin\left(\Log z\right) \right).
		\end{aligned}
		\end{equation}
		In addition, we have
		\begin{equation}
		\begin{aligned}
		\abs{5\pi+2\pi i\cos\left(\Log (x+iy)\right)}\geq \Re\left(5\pi+2\pi i\cos\left(\Log (x+iy)\right)\right)
	%	&=5\pi+2\pi\sin\left(\Log\left(\sqrt{x^2+y^2}\right)\right)\cdot \sinh\left(\arctan\left(\dfrac{y}{x}\right)\right) \\
	%	&\geq 5\pi - 2\pi \sinh\dfrac{\pi}{3} \\
		\geq \dfrac{5\pi}{2}
		\end{aligned}
		\end{equation} 
		and
		\begin{equation}
		\begin{aligned}
		\abs{\exp\left(2\pi i \sin\left(\Log (x+iy)\right)\right)}=\exp\left(-2\pi \Im\left(\sin\left(\Log (x+iy)\right)\right)\right) %\\
		%&=\exp\left(-2\pi\cos\left(\Log\left(\sqrt{x^2+y^2}\right)\right)\cdot \sinh\left(\arctan\left(\dfrac{y}{x}\right)\right)\right) \\
		\geq \exp\left(-3\pi\right).
		\end{aligned}
		\end{equation}
		Together this yields for large $z\in W$ that
		\begin{equation}
		\begin{aligned}
		\abs{g'(z)}\geq \dfrac{5\pi}{2} e^{-3\pi}\cdot \abs{z}^{5\pi-1} \geq 1.
		\end{aligned}
		\end{equation}
		With \eqref{eq:C_3} and \eqref{eq:f_in_B} we obtain
		\begin{equation}
		\begin{aligned}
		\abs{f(\xi)} \leq \abs{\tilde{f}(\xi)}+\abs{\exp(g(\xi))} \leq C_3 + \exp(\Re g(\xi)) \leq C_3+\dfrac{C_3}{\kappa \abs{g'(\xi)}} \leq C_3+\dfrac{C_3}{\kappa}
		\end{aligned}
		\end{equation}
		Since $f$ is of finite order, the number of asymptotic values is finite by Lemma \ref{DCA thm}. Thus $\sing(f^{-1})$ is bounded, i.e.\ $f\in\mathcal{B}$.
	\end{proof}
	\section{Proof of Theorem \ref{THM2}}
	\begin{proof}[Proof of Theorem \ref{THM2}]
		From Lemma \ref{Konstruktion f} we have obtained a suitable function $f$ which is close to $e^g$ in the subset $A$ of our desired tract. The function we will construct in the following will be of disjoint type. 
		
		First choose $R>0$ large enough such that $\sing(f^{-1}) \subset D(0,e^R)$.
		Following Remark \ref{f_0=lambda f}, we choose $\lambda$ small enough to obtain all features of $f_0$ mentioned there with $D=D(0,e^R)$.
		
		By Theorem \ref{RRRS} we know that $\J(f_0)$ consists of an uncountable union of pairwise disjoint hairs. We now want to show that every hair of $f_0$ is nowhere differentiable including all endpoints. 
		
		For $\tilde{A}:=f_0^{-1}\left(\C\setminus \overline{\D}\right)$ denote $T:=\Log\big(\tilde{A}\big)$ and define $\mathcal{T}$ as in Remark \ref{Tracts_log_coord}. 
		Thus the logarithmic transform of $f_0$ takes the form $F\colon \mathcal{T}\to \mathbb{H}_{> 0}:=\{z\in \C : \Re z >0\}$. By Remark \ref{Baranski2} we know that $\J(F)$ also consists of an uncountable union of pairwise disjoint hairs. From now on we fix an address $\ul{s}$ such that the set $\{z\in J(F) : z \text{ has address }  \ul{s}\}$ is non-empty and thus a hair by Theorem \ref{Hairs}. We consider an arbitrary parametrization $\gamma\colon [0,\infty)\to\C$ of this hair. Then $F^n(\gamma)$ is the hair in $J(F)$ having the address $s_{n}s_{n+1}\dots$  with endpoint $F^n(\gamma(0))$. By Remark \ref{Re F} we have $\Re F^n|_{\gamma\left((0,\infty)\right)} \to \infty$ locally uniformly as $n\to\infty$.
		It follows from the construction of the tract $T\supset T(0)$ and the fact that $F^n(\gamma)$ has address $s_{n}s_{n+1}\dots$ for all $n\in\N$ that $t \mapsto\Im (F^n(\gamma(t))-2\pi i s_n)$ changes its sign infinitely often. Remark \ref{tract_constants} yields that their exists for all $n\in\N_0$ a sequence $\left(z_{n,k}\right)_{k\in\N}\in \left(F^n(\gamma((0,\infty))\right)^\N$ such that the following holds:
		\begin{enumerate}
			\item We have for all $n\in\N_0$
			\begin{equation}
			\begin{aligned}
			\Im \left(z_{n,k}-2\pi is_n\right)
			\begin{dcases}
			>1/5, & \text{if } k \text{ is odd},\\
			<-1/5, & \text{if } k \text{ is even}. \\
			\end{dcases}
			\end{aligned}
			\end{equation}
			\item For all $n\in\N_0$ and $k\in \N$ we have $\Re z_{n,k+1}-\Re z_{n,k} =\pi$ and thus $\lim\limits_{k\to\infty} \Re z_{n,k}=\infty$.
		\end{enumerate}
		Let $w_0\in \gamma([0,\infty))$ and denote $w_n:=F^n(w_0)$ for all $n\in\N$. In this case we will take $z_{n,k}:=x_{n,k}+iy_{n,k}$, where $x_{n,k}$ can be chosen to be $\pi/2 + \pi\left(\lceil \Re F^n(\gamma(0))/\pi  \rceil+ k\right)$ according to Remark \ref{tract_constants}. Thus the imaginary part of $z_{n,k}-2\pi is_n$ is bounded uniformly away from $0$.
		
		Since $F$ is the logarithmic transform of the function $f_0$ of disjoint type, we have $\overline{\mathcal{T}}\subset \mathbb{H}_{\geq 0}$ and  there exists by Remark \ref{f_0=lambda f} a constant $K>1$ such that we have $\abs{F'(z)}\geq K$ for all $z\in \mathcal{T}$. Thus every branch $F^{-1}_k \colon \mathbb{H}_{>0} \to T(s_k)$ of $F^{-1}$ is a uniform contraction on $\mathbb{H}_{>0}$.
		
		Now we proceed as follows. If $\Im (w_n-2\pi is_n) >0$, there exists a minimal $k_n\in 2\N$ such that 
		\begin{equation}
		\begin{aligned}
		\Re w_n < x_{n,k_n} \quad \text{and} \quad y_{n,k_n}-2\pi s_n<0
		\end{aligned}
		\end{equation}
		and consider the triangle $\Delta(w_n,z_{n,k_n},z_{n,k_n+1})$ (see Figure \protect \ref{Dreieck_hair}). Since for all $n\in\N$
		\begin{equation}
		\begin{aligned}
		\dfrac{\pi}{4} < x_{n,k_n} - \Re w_n < 2\pi
		\end{aligned}
		\end{equation}
		and
		\begin{equation}
		\begin{aligned}
		\dfrac{1}{5} < \Im w_n- y_{n,k_n}< \dfrac{3}{2}
		\end{aligned}
		\end{equation}
		holds by Remark \ref{tract_constants}, we obtain for all $n\in\N$
		\begin{equation}
		\begin{aligned}
		\dfrac{1}{20\pi} < \abs{\dfrac{\Im w_n - y_{n,k_n}}{\Re w_n - x_{n,k_n}}} < \dfrac{6}{\pi}.
		\end{aligned}
		\end{equation}
		Using the fact that $x_{n,k_n+1}- x_{n,k_n}=\pi$, we obtain 
		\begin{equation}
		\begin{aligned}
		\dfrac{2}{5\pi} < \abs{\dfrac{ y_{n,k_n}- y_{n,k_n+1} }{x_{n,k_n}- x_{n,k_n+1}}} < \dfrac{3}{2\pi}.
		\end{aligned}
		\end{equation}
			\begin{figure}[H]
				\begin{center}
					\includegraphics[height=8.5cm]{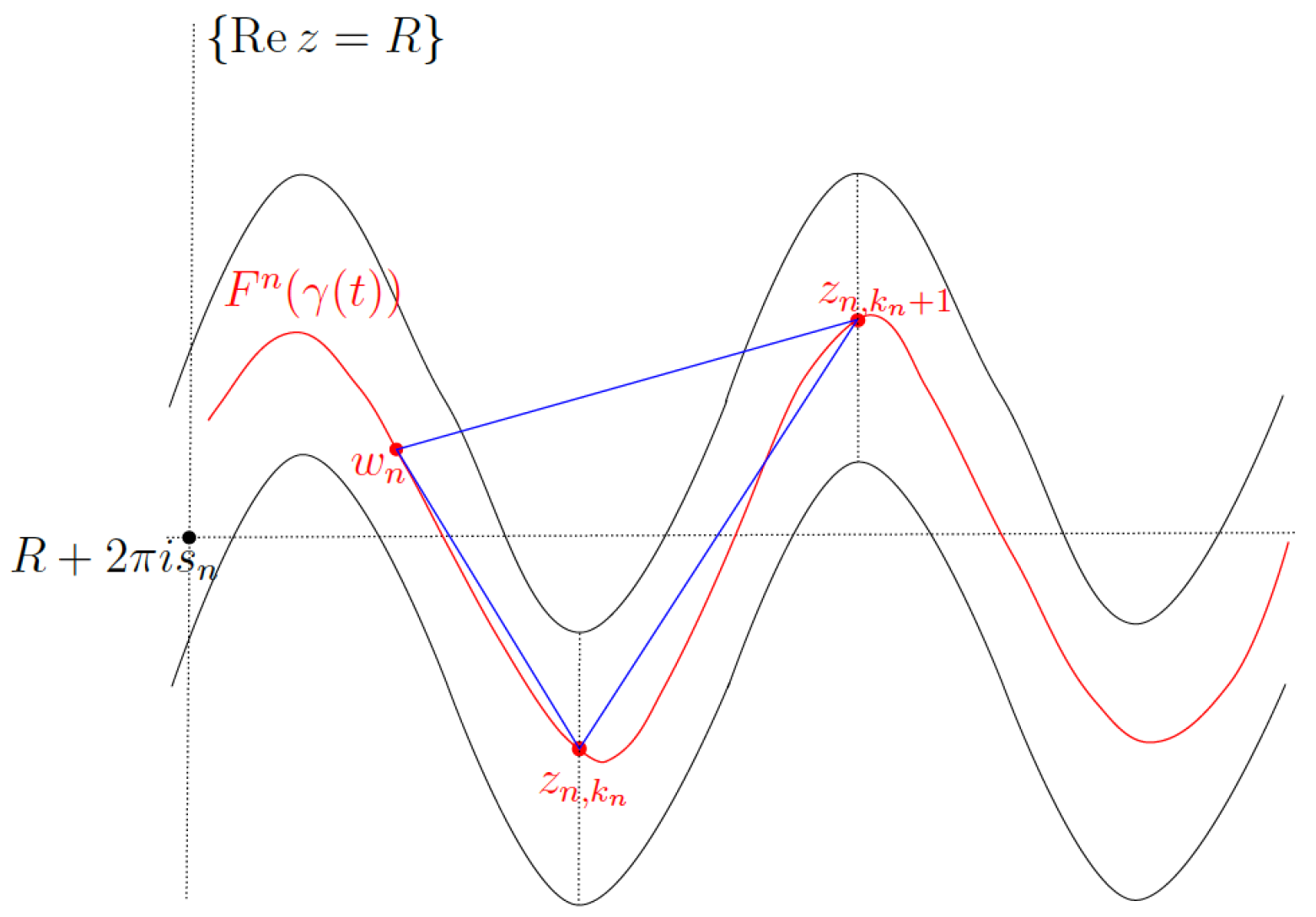}
					\caption{Construction of the function $g$ using the logarithmic change of variable of the map $h$.}
				\end{center}
			\end{figure}
		Hence 
		\begin{equation}
		\begin{aligned}
		\dfrac{1}{20\pi}<\arctan\left(\abs{\dfrac{\Im w_n - y_{n,k_n}}{\Re w_n - x_{n,k_n}}}\right) + \arctan\left(\abs{\dfrac{ y_{n,k_n}- y_{n,k_n+1} }{x_{n,k_n}- x_{n,k_n+1}}}\right)< \dfrac{\pi}{2}
		\end{aligned}
		\end{equation}
		and thus the angle at $z_{n,k_n}$ is always greater than $\pi/2$ and less than $19\pi/20$. This yields that $\Delta(w_n,z_{n,k_n},z_{n,k_n+1})$ is non-degenerate, say
		\begin{equation}
		\begin{aligned}
		\mathcal{D}\left(\Delta(w_n,z_{n,k_n},z_{n,k_n+1})\right) < 1-\delta \label{eq:Dreieck_Endbeweis}
		\end{aligned}
		\end{equation}
		for a suitable $\delta >0$.
		
		Analogously, if $\Im (w_n-2\pi is_n) <0$, there exists $k_n\in 2\N+1$ with the same property. Since $4\pi > \sqrt{(3\pi)^2+(3/4)^2}$ and by our construction of the tract we know that the disc $D_n:=D(w_n,4\pi)$ contains the triangle $\Delta(w_n,z_{n,k_n},z_{n,k_n+1})$ for all $n\in\N$. If we choose $R>4\pi$ we can guarantee that $\Re w_n-4\pi>0$ for all $n\in \N_0$ since $\Re w_n > R$. Now we want to show that with
		\begin{equation}
		\begin{aligned}
		\phi_n := F^{-1}_0 \circ \dots \circ F^{-1}_{n-1}
		\end{aligned}
		\end{equation} 
		the map $\phi_n|_{D_n}$	has a uniformly bounded distortion. If we put $R_n:=\dist(w_n,\partial \mathbb{H}_{> 0})$ for all $n\in\N_0$, we obtain $R_n=\Re w_n > R$ for all $n\in \N_0$. Since $\phi_n$ is conformal on $D(w_n,R_n)$, we obtain by Remark \ref{Koebe_Verzerrung}
		\begin{equation}
		\begin{aligned}
		\mathcal{L}\left(\phi_n|_{D_n}\right) \leq \left(\dfrac{R_n+4\pi}{R_n-4\pi}\right)^4 \leq  \left(\dfrac{R+4\pi}{R-4\pi}\right)^4 \leq 1+\delta
		\end{aligned}
		\end{equation}
		for $R$ large enough. To guarantee convexity of $\phi_n(D_n)$, we need by Lemma \ref{Radius of convexity} to take $R$ large enough such that $2\pi /R < \sqrt{2}-1$. Together with Lemma \ref{Kleine_Verzerrung} and \protect \eqref{eq:Dreieck_Endbeweis} this yields
		\begin{equation}
		\begin{aligned}
		\mathcal{D}\left(\Delta\left(\phi_n(w_n),\phi_n(z_{n,k_n}),\phi_n(z_{n,k_n+1})\right)\right)
		& \leq (1+\delta)\cdot \mathcal{D}\left(\Delta(w_n,z_{n,k_n},z_{n,k_n+1})\right) \\
		&< 1-\delta^2.
		\end{aligned}
		\end{equation}
		
		Since $\abs{F'(z)}\geq K$ for all $z\in \mathcal{T}$ and $\left(F^{-1}_j\right)'(z) = \left(F^{-1}_0\right)'(z)$ for all $j\in\N_0$, we obtain by the chain rule
		\begin{equation}
		\begin{aligned}
		\abs{\phi'_n(z)} = \dfrac{1}{\abs{(F^n)'\left(F^{-n}_0(z)\right) }} 
		= \dfrac{1}{\prod_{j=0}^{n-1}\abs{F'(F^{j-n}_0(z))}} 
		\leq \dfrac{1}{K^{n}}.
		\end{aligned}
		\end{equation}
		Since the distortion of $\phi_n$ on the triangle $\Delta(w_n,z_{n,k_n},z_{n,k_n+1})$ is bounded for all $n\in \N_0$, these triangles converge to the single point $w_0$ for $n\to\infty$. On the other hand, these triangles are always non-degenerate.
		Thus $\gamma$ is not differentiable in $w_0$ by Lemma \ref{gamma_Dreieck}. Since $w_0$ was arbitrary, the claim follows.
	\end{proof}

\end{document}